\documentclass[11pt]{amsart}
\usepackage{amsmath,amsthm,verbatim,amssymb,amsfonts,amscd,  graphicx}
\usepackage[normalem]{ulem}
\usepackage{xcolor}
\usepackage{graphics}

\usepackage{euscript, enumerate}
\usepackage{url,hyperref}

\usepackage{hyperref}

\usepackage{amscd}

\renewcommand{\thispagestyle}[1]{}
\usepackage{color}
\usepackage{geometry}
\newtheorem{prop}{Proposition}[section]
\newtheorem{theorem}[prop]{Theorem}
\newtheorem{lemma}[prop]{Lemma}
\newtheorem{corollary}[prop]{Corollary}
\newtheorem{remark}[prop]{Remark}

\newtheorem{definition}[prop]{Definition}
\newtheorem{conjecture}[prop]{Conjecture}
\newtheorem{problem}[prop]{Problem}

\def\begeq{\begin{equation}}
\def\endeq{\end{equation}}

\providecommand{\keywords}[1]
{
  \small	
  \textbf{\textit{Keywords---}} #1
}

\title{Geometry of Positive Scalar curvature on Complete Manifold}
 \author{Bo Zhu}
\address[Bo Zhu]{School of Mathematics, University of Minnesota-Twin Cities, MN 55455, USA.}
\email{zhux0629@umn.edu}
\date{\today}

\subjclass[2020]{Primary 53C21
}

\keywords{Positive scalar curvature, Volume growth, Yau's problem, Integral of scalar curvature, Uryson width}
\begin{document}
\maketitle
\begin{abstract}
In this paper, we study the interplay of geometry and positive scalar curvature on a complete, non-compact  manifold with non-negative Ricci curvature. In three-dimensional manifold, we prove a minimal volume growth, an estimate of integral of scalar curvature and width. In higher dimensional manifold, we obtain a volume growth with a stronger condition.
\end{abstract}

\section{Introduction}
An important topic in geometric analysis is to understand the interplay between curvature and geometry. One of the classical and widely popularized results in this aspect is the Myers and Cheng's maximal diameter theorem.

\begin{theorem}[Maximal Diameter Theorem, \cite{cheng,myers}]\label{mdt}
 Any complete Riemannian manifold $(M^n,g) $ with the Ricci curvature $Ric(g) \geq n-1$ has $\mathrm{Diam}(M) \leq \pi$ with equality if and only if $M$ is a round sphere. 
\end{theorem}

From the perspective of size geometry  \cite{msg}, i.e., diameter, volume,  Uryson width, Filling Radius, injectivity radius, etc. are called the size quantities of a Riemannian manifold. Myers and Cheng's maximal diameter theorem indicates that the positivity of Ricci curvature completely controls its distance spread in all directions, i.e., diameter. Here, we primarily focus on the size of Riemannian manifold and metric structure of Riemannian manifold. A natural problem is that how we could generalize theorem \ref{mdt} to the scalar curvature in some sense. It is clear that positive scalar curvature on a Riemannian manifold can not determine its own distance spread fully. For instance,  $(\mathbb{S}^{2}\times \mathbb{R}^{n-2},g)$ has scalar curvature $2$ but no control of the diameter if $g$ is the standard direct product of Riemannian metric. Based on this basic example, we could never expect that the positivity of scalar curvature on a Riemannian manifold can fully control its size. In fact, the most promising expectation is that the positivity of scalar curvature on a Riemannian manifold should have  control on the size, which is only related to $1$ or $2$ dimensional quantities. In this direction, Gromov conjectures that 

\begin{conjecture}[Gromov, \cite{gromovpsc}]\label{gromovdim}
Let $(M^n,g)$ be a complete non-compact manifold with the scalar curvature $Sc(g) \geq n(n-1)$. Then the macroscopic dimension of $M$ satisfies
$$macrodim(M) \leq n-2.$$ 
\end{conjecture}

It remains open even for $marcodim(M) \leq n-1$.  In the view of geometric dimension theory, Conjecture \ref{gromovdim} is equivalent to the statement that $M$ can be approximated by $n-2$ dimensional polyhedrons within finite distance. For the definition of macroscopic dimension, the reader can refer to the references \cite{msg,gromovpsc}, and we will not use the definition of macroscopic dimension directly. However, in the spirit of it, Conjecture \ref{gromovdim} provides the insight for the following conjectures and the results of our paper.
\vskip 2mm

Together with Theorem \ref{mdt}, it is proved that for the conjugate radius $\mathrm{conj}(M)$ of a Riemannian manifold $M$,
\begin{theorem}\label{green}[Maximal Conjugate Radius Theorem,  \cite{green}]
Let $(M^n,g)$ be a closed Riemannian manifold with the scalar curvature $Sc(g) \geq n(n-1)$. Then
$
    \mathrm{conj}(M) \leq \pi,
$
and equality holds if and only $M$ is isometric to the round sphere $\mathbb{S}^n$.
\end{theorem}
Since  the injectivity radius $\mathrm{inj}(M) \leq \mathrm{conj}(M)$, we obtain that Theorem \ref{green}  implies that for any  closed Riemannian manifold with $Sc(g) \geq n(n-1)$,
$
    \mathrm{Inj}(M) \leq \pi,
$
and equality holds if and only $M$ is isometric to the round sphere $\mathbb{S}^n$, which we would like to name as maximal injectivity radius theorem. For the details of the proof of Theorem \ref{green}, see \cite{green, gk}.

To my best knowledge, it remains open whether maximal conjugate or injectivity radius still hold on any complete non-compact Riemannian manifold. In fact, this question is deeply related to Conjecture \ref{gromovdim}. Here, we apply the techniques used in the case of closed Riemannian manifold to the complete,  non-compact Riemannian manifold $(M^n,g)$ and then obtain a local estimate on the integral of scalar curvature.

\begin{prop} \label{completegreen}
Let $(M^n, g)$ be a complete, non-compact Riemannian manifold. Then,
\begin{itemize}
    \item Suppose that $B(p, R) \subset M$ is a geodesic ball with center $p \in M$ and radius $r >0$.  If $Ric(g) \geq 0$ on $B(p,R)$, then we obtain,
\begin{equation}\label{conjugateiscalarintegral}
 \int_{B(p,R-c)} Sc  \leq n(n-1)(\frac{\pi}{c})^2\mathrm{vol} B(p,R), \forall  c \leq  \mathrm{conj}(M);
\end{equation}
\item If $Sc(g) \geq n(n-1)$ on $M$ and $Ric(g) \geq 0$ on $M$, then $\mathrm{inj}(M) \leq \mathrm{conj}(M) \leq \pi$.
\end{itemize}

\end{prop}

The conjugate radius estimate indicates that positive scalar curvature does imply that a Riemannian manifold is curved or becomes thinner under the assumption of non-negative Ricci curvature and strictly positive scalar curvature. However, we still do not know whether Proposition \ref{completegreen} holds without any assumptions on non-negative Ricci curvature. On the one hand, if one can construct a complete, non-compact Riemannian manifold with positive scalar curvature and its injectivity radius is infinity, then it will deduce a negative answer to Conjecture \ref{gromovdim}; on the other hand, the local estimate (\ref{conjugateiscalarintegral}) indicates that the average of the integral of scalar curvature is bounded above in terms of the lower bound of the conjugate radius. Also, it implies that the volume growth and positive scalar curvature are intertwined locally. However, the interplay is still a mystery on a global scale. Many years ago, Yau proposed the following problem, which is involved with the volume growth and positivity of scalar curvature on a complete, non-compact Riemannian manifold with non-negative Ricci curvature.

\begin{problem}[Yau \cite{Yau}]\label{Yau}
Let $(M^n, g)$ be a complete, non-compact manifold with non-negative Ricci curvature and $B(p,r) \subset M$ a geodesic ball with center $p \in M$ and radius $r$. Do we have
\begin{equation}\label{yauscalar}
 \limsup_{r  \rightarrow \infty }r^{2-n}\int_{B(p,r)} Sc < \infty ?
\end{equation}
\end{problem}

In fact, Yau proposed a more general version of this problem that is involved with the $\sigma_k, k=1,2, \cdots, n$ of Ricci tensor in \cite{Yau}. Unfortunately, Yang \cite{yb} constructs a counterexample on K\"{a}hler manifold to prove that the general version of Yau's Problem \ref{Yau} does not hold for $k=1, 2, \cdots , n-1$, Xu \cite{Xu} obtains an estimate involved with the integral of scalar curvature towards the Problem \ref{Yau} in the case of three-dimensional Riemannian manifold by using the monotonicity formulas of Colding and Minicozzi \cite{cmmonotonicity}.  However, Problem \ref{Yau} remains open. In fact, it has been shown \cite{pet} that the inequality (\ref{yauscalar}) holds if we impose a strong curvature condition non-negative sectional curvature instead of non-negative Ricci curvature. Also, Naber \cite{naberscalar} asks the non-collapsing version of Yau's Problem \ref{Yau} that is a baby version, and propose a local version of Yau's Problem \ref{Yau}. Here,  we propose a baby version of Yau's Problem \ref{Yau} that is worthwhile of investigating as well.

\begin{problem}\label{pscvol}
Suppose that $(M^n,g)$ is a complete, non-compact Riemannian manifold with $Ric(g) \geq 0$ and $Sc(g) \geq 1$. Do we have
$$\limsup_{ r \rightarrow \infty }\frac{\mathrm{vol}(B(p,r))}{r^{n-2}} < \infty ?$$
\end{problem}

On the one hand, this problem can be regarded as a baby version of Yau's Problem \ref{Yau}: if Problem \ref{Yau} holds, then Problem \ref{pscvol} holds;  on the other hand, from the perspective of size geometry, this problem can provide more valid evidence for Gromov's Conjecture \ref{gromovdim}, and it can be considered as a quantitative version of \ref{gromovdim} in the category of Riemannian manifolds with non-negative Ricci curvature. 

Here, this study primarily focuses on the case of three-dimensional, complete, non-compact Riemannian manifolds. In fact, there are abundant of studies on three-dimensional Riemannian manifolds, including the case of compact or non-compact. For the closed three-dimensional Riemannian manifolds, the topological classification of three-dimensional Riemannian manifolds is clear according to Poinc\'{a}re Conjecture. Moreover,  the proof of Thurston’s Geometrization conjecture \cite{p1,p2,p3} shows that a closed three-dimensional Riemannian manifold admits a metric with positive scalar curvature if and only if it is a connected sum of spherical 3-manifolds
and some copies of $S^1\times S^2$. For complete non-compact three-dimensional Riemannian manifolds with non-negative Ricci curvature, Liu \cite{LG} proves that it is either diffeomorphic to $\mathbb{R}^3$ or its  universal cover splits. Therefore, there are not many three-dimensional Riemannian manifolds with the properties that they admit a complete Riemannian metric with non-negative Ricci curvature and positive scalar curvature. Topologically, they are either $\mathbb{R}^3$ or $\mathbb{S}^2 \times \mathbb{R}$. In a recent progress, Wang \cite{wj} shows that any complete, non-compact contractible three-dimensional Riemannian manifold with non-negative scalar curvature is homeomorphic to $\mathbb{R}^3$. M-W \cite{mw} obtain a comparison theorem of positive Green function and spectrum estimate on complete, non-compact, three-dimensional Riemannian manifold with some extra topological assumptions.

However, our goal in the paper is to study the geometry of a complete, non-compact, three-dimensional Riemannian manifold that positive scalar curvature has influence on, rather than the topology of a complete, non-compact, three-dimensional Riemannian manifold. According to the splitting theorem of complete, non-compact Riemannian manifold with non-negative Ricci curvature \cite{cg}, we primarily focus on the geometry of positive scalar curvature on $\mathbb{R}^3$.

\bigskip 

First, we have the following observation on Yau's Problem \ref{Yau}
\begin{theorem}\label{polecase}
Let $(M^3,g)$ be a complete, non-compact three-dimensional Riemannian manifold with a pole $p$ and $Ric(g) \geq 0$. Then
\begin{equation}
    \limsup_{r \rightarrow \infty } \frac{1}{r} \int_{B(p,r)}Sc \leq  20\pi. 
\end{equation}

\end{theorem}
This estimate confirms that Yau's Problem \ref{Yau} holds in the special case of complete, non-compact three-dimensional Riemannian manifolds with a pole. But, our assumption on the manifolds with pole is very artificial. For Yau's Problem \ref{Yau}, we will do more studies in our future works. 

\vskip 5mm

Moreover, we consider the baby version of Yau's Problem \ref{Yau}, then we obtain

\begin{theorem}\label{mainvolumegrowth}
Let $(M^3,g)$ be a complete, non-compact, three-dimensional Riemannian manifold with $Ric(g) \geq 0$ and $Sc(g) \geq 6$. Then, for any $p \in M$, we obtain
\begin{equation}
\limsup_{R \rightarrow \infty} \frac{\text{vol}(B(p,R))}{R} < \infty,
\end{equation}
if $\text{vol}(B(q,1)) \geq \epsilon >0$ for any $q \in M$.
\end{theorem}

Yau \cite{Yauharmonic} proved that any complete, non-compact manifolds
with non-negative Ricci curvature have at least linear volume growth. Combining Yau's result with our Theorem \ref{mainvolumegrowth}, we obtain that any complete, non-compact three-dimensional manifold
with non-negative Ricci curvature and strictly positive scalar curvature has linear volume growth or, namely, minimal volume growth.

\vskip 5mm

As a byproduct of the proof of Theorem \ref{mainvolumegrowth}, we obtain the following volume estimate in higher dimensional case,

\begin{corollary}\label{bnvolumegrowth}
Let $(M^n,g)$ be a complete, non-compact Riemannian manifold with $Ric(g) \geq 0$ and $Sc(g) \geq n(n-1)$. Then for any $p \in M$, we obtain
\begin{itemize}
\item If $vol(B(q,1)) \geq \epsilon >0$ for all $q \in M$, then
$$\limsup_{R \rightarrow \infty} \frac{\text{vol}(B(p,R))}{R^{n-1}} < \infty;$$

\item if $\mathrm{inj}(M) \geq \epsilon >0$, then
$$\limsup_{R \rightarrow \infty} \frac{\text{vol}(B(p,R))}{R^{n-2}} < \infty.$$

\end{itemize}

\end{corollary}

In fact, Anderson proved that \cite{anderson} any complete, non-compact Riemannian manifold $(M^n,g)$ with positive Ricci curvature has $b_1(M) \leq n-3$  and the rank of any free Abelian subgroup of $\pi_1(M)$ is at most $n-3$ for which the estimate is optimal. From Corollary \ref{bnvolumegrowth}, we also expect  that non-negative Ricci curvature and strictly positive scalar curvature on a complete, non-compact Riemannian manifold would imply $b_1(M) \leq n-3$ and it is optimal as well. For relevant result, you may refer to \cite{coldingtorus}. In fact, it will be interesting to study the dimension of harmonic functions with linear growth on a complete, non-compact manifold with non-negative Ricci curvature and strictly positive scalar curvature on higher dimensional Riemannian manifolds.

\begin{remark} In fact, Gromov \cite{gromovlarge} stated that, under the assumption that $K(g) \geq 0, Sc(g) \geq n(n-1)$ without any details, then 
\begin{equation}\label{volumesection}
    \sup_{p \in M}vol(B(p,r)) \leq c_n r^{n-2}.
\end{equation}
For the proof of $(\ref{volumesection})$, you may refer to \cite{pet}. Furthermore, Gromov \cite{gromovlarge} conjectures that the volume estimate $\ref{volumesection}$ holds if we only $Ric(g) \geq 0$. Here, our results can be regarded as a step to prove Gromov's conjecture under an extra condition volume non-collapse as $n=3$ or injectivity radius non-collapse as $n \geq 4$.
\end{remark}

Moreover, a complete, non-compact Riemannian manifold is said to be non-parabolic if it admits a positive Green function, otherwise it is said to be parabolic. By a result of Varopoulos \cite{var}, a complete, non-compact Riemannian manifold with $Ric(g)\geq 0$ is non-parabolic if and only if
$$\int_1^\infty \frac{r}{vol(B(p,r))} dr< \infty.$$
Hence, the following conclusion is deduced
\begin{corollary}
Let $(M^3,g)$ be a complete, non-compact three-dimensional Riemannian manifold with $Ric(g) \geq 0$, $Sc(g) \geq 6$ and $\text{vol}(B(p,1)) \geq \epsilon >0$ for all $p \in M$. Then $(M^3,g)$ is  parabolic.
\end{corollary}

According to the result in \cite{S}, $(M^3,g)$ admits no any nontrivial harmonic functions with polynomial growth unless it splits. In fact, we would like to believe that: on a complete, non-compact, $n$-dimensional Riemannian  manifold with non-negative Ricci curvature and strictly positive scalar curvature, then the dimension of harmonic function with linear growth should be less or equal to $n-2$ for any $n \geq 4$.

\bigskip

Finally, let's study the width of the manifold. Sormani \cite{sdiameter} proves for any complete non-compact manifold with non-negative Ricci curvature and minimal volume growth, Busemann function is proper. However, we can not expect that the diameter of  the level set of Busemann function has a uniformly bound there and many examples are illustrated in \cite{sdiameter}. Here, we prove an upper bound on the width of the manifold.

\begin{theorem}\label{width}
Let $(M^3,g)$ be a complete, non-compact three-dimensional Riemannian manifold with $Ric(g) \geq 0$, $Sc(g) \geq 2$ and $\text{vol}(B(p,1)) \geq \epsilon >0$. Then, there exists a constant $c$ and continuous function $f: M \rightarrow \mathbb{R}$ such that for any $r \in \mathbb{R}$, 
$$ \text{diam}(f^{-1}(r)) \leq c.$$
\end{theorem}

\begin{remark}
Theorem \ref{width} indicates that any 3 dimensional manifold $(M^3,g)$ with non-negative Ricci curvature and positive scalar curvature will wander around a line. Hence, in the large scale, $M$ is a one dimensional line $\mathbb{R}$. In fact, by the proof of Theorem \ref{width}, we know $f$ can be a Lipschitz function such that $Lip(f) \leq 1$. Besides, it is possible that $f^{-1}(r)$ for some $r$ may not be connected, and there is no reason to expect that $f^{-1}(r)$ for all $t \in \mathbb{R}$ are connected. Hence, for the case of discounted level set,  we also count the distance among different connected components of the level set in the theorem.
$$diam(f^{-1}(r))= \sup_{x,y}\{d_g(x,y): f(x)=f(y)=r\}.$$
Here, $d_g$ is the distance induced by the Riemannian metric $g$ on $M$. Finally, by the language of Uryson width, Theorem \ref{width}  implies that $width_1(M) \leq c$ and $macrodim(M^3) =1$. 

\end{remark}

\vskip 5mm

The article is organized as follows. In section 2, we introduce some preliminary materials for the proofs of main theorems and then prove Proposition \ref{completegreen}. In section 3, we prove Theorem \ref{polecase}, \ref{mainvolumegrowth}, Corollary \ref{bnvolumegrowth} and Theorem \ref{width}. Now, let's briefly outline the proof of main results in this work. For the proof of Theorem \ref{polecase}: first, by using the stability of the geodesic ray, we deduce an upper bound on the integral of Ricci curvature in the direction of the normal vector field over the geodesic ball centered at the pole. Then, we make use of the geometrically relative Bochner formula and Gauss Bonnet formula on a geodesic sphere to obtain an upper bound on the integral of scalar curvature. In fact, the geometrically relative Bochner formula plays a vital role in the proof to get rid of the second fundamental form of the geodesic sphere and Theorem \ref{polecase} will be much stronger than Proposition \ref{completegreen}. For the proof of Theorem \ref{mainvolumegrowth}: we first prove Lemma \ref{independentgeodesic} which relates volume growth to the Euclidean split of Ricci limiting space. Then, we argue our theorem by contradiction. By combining Lemma \ref{independentgeodesic} with Cheeger-Colding theory and the result of M-T on the characterization of three-dimensional Ricci limiting space, we obtain that, topologically, the Ricci limiting space can only be: $\mathbb{R}^3, \ \mathbb{R}^2 \times S^1$, then, we rule out both cases by using the torus band estimate. Hence, we prove a minimal volume growth under strictly positive scalar curvature. Here, we avoid defining a generalized scalar curvature on the Ricci limiting space, which is a hard question for the author. Instead, we use the Lipschitz structure of positive scalar curvature proved by G-WXY to rule out $2$ cases of Ricci limiting space. Essentially, our argument works whenever the Ricci limiting space is a topological manifold and hence we have a corresponding result in the higher dimensional Riemannian manifolds with stronger condition assumption. In fact, the method of using Lipschitz structure of strictly positive scalar curvature is quite new to study the positivity of scalar curvature on complete Riemannian manifolds with non-negative Ricci curvature. 

\bigskip 

Acknowledgement: The author would like to express his gratitude to his advisor Prof. Jiaping Wang for his constant support. The author is also grateful to Prof. Ruobing Zhang, Prof. Wenshuai Jiang for their help on metric geometry.  The author is indebted to Prof. Hanlong Fang and Pak-Yeung Chan, Zhichao Wang for their discussions over the past years. 
\bigskip


\bigskip

\section{Preliminaries and Notations}

In this section, let's make some preparations and prove some lemmas for the proof of the main theorems in the paper.  We will start by the stability of the geodesics in a Riemannian manifold.

\subsection{Variation of Geodesic}

Suppose that $(M^n,g)$ is a complete Riemannian manifold and $\gamma: [a,b] \rightarrow M$ is a smooth curve in $M$ with $\|\gamma^{\prime}(t)\| =1$ and $\gamma(a)=p, \gamma(b)=q, \ p,q \in M$. Then, we consider a smooth variation of  $\gamma(t)$:
$$\gamma(t,s): [a,b] \times [-\epsilon, \epsilon] \rightarrow M. $$
and $\gamma(t,0)=\gamma(t)$ and $\gamma(a,s)=p$ and $\gamma(b,s)=q$. We say that $\gamma$ is a geodesic if $\gamma$ is a critical point of the length functional $$L(s)= \int_a^b \|\frac{\partial \gamma(t,s) }{\partial t }\|dt.$$ That is for any variation vector field $X(t)$ and $X(t)= \frac{\partial \gamma(t,s)}{\partial s}|_{s=0} $ with $X(a)=X(b)=0$, we have

\begin{equation}
   0 = L^\prime(0) = \int_a^b \langle \gamma^{\prime\prime}(t), X(t) \rangle dt.
\end{equation}
It is equivalent to saying that $\gamma^{\prime\prime}(t)=0$ on $[a,b]$.

Moreover, we calculate the second variation of the arc length functional on geodesic $\gamma$, then

\begin{equation}\label{stable}
    L^{\prime\prime}(0)= \int_a^b \|\nabla X\|^2 -  \langle R(X, \gamma^\prime(t))\gamma^\prime(t), X\rangle dt.
\end{equation}
Here, $\langle R(X, \gamma^\prime(t))\gamma^\prime(t), X\rangle$ is the Riemannian curvature and $\nabla$ is the Levi-Civita connection on $(M,g)$. A geodesic is said to be stable if the second variation is non-negative. i.e.,  $L^{\prime\prime}(0)\geq 0$. Since the calculations above are classical and standard on any Riemannian geometry textbook, we omitted the details (See \cite{ligeometryanalysis}).

Then, for any fixed point $x \in M$,  we introduce the exponential map as introduce
$$exp:  T_xM \rightarrow M, exp(v) = \gamma(1), v \in T_xM.$$

\begin{definition} Let $(M^n,g)$ be a complete Riemannian manifold. Then we define the injectivity radius of $M$ as follows
$$
\mathrm{Inj}(M):= \inf_{x \in M} \sup_{r}\{r: exp: B(x,r)\rightarrow exp(B(x,r)) \ \text{is a diffeomorphism}\}.$$
and the conjugate radius of $M$ as follows
$$\mathrm{conj}(M) = \inf_{x} \sup_r \{r: exp: B(x,r) \rightarrow exp(B(x,r)), \ B(x,r) \ \text{contains no critical points of exp}.\}$$
Finally,
$(M^n,g)$ is said to be a manifold with a pole at $p$ if $exp_p: T_p M \rightarrow M$ is a diffeomorphism.
\end{definition}
By the definition of injectivity radius and conjugate radius of $(M,g)$, we have
$\mathrm{inj}(M) \leq \mathrm{conj}(M)$. Moreover, let $(M^n,g)$ be a complete Riemannian manifold and $\gamma_{x,v}(t)$  the unique geodesic with initial conditions 
$$\left\{
\begin{array}{l}
     \gamma_{(x,v)}(0)=x;\\
     \gamma_{(x,v)}^\prime(0)=v, v \in T_xM.
\end{array}
\right.$$
The initial problem is solvable uniquely by the classical ODE problem and hence $\gamma$ always exists.
\begin{definition}
For a given $t\in \mathbb{R}$, we define a diffeomorphism of the tangent bundle $TM$
$$\varphi_t: TM \rightarrow TM.$$
as follows
$$\varphi_t(x,v)= (\gamma_{(x,v)}(t), \gamma_{(x,v)}^\prime(t) ).$$
\end{definition}
In fact, the family of diffeomorphism $\varphi_t$ is a flow. i.e., it satisfies with $\varphi_{t+s}=\varphi_t\circ \varphi_s$ for any $t,s \in \mathbb{R}$ since the uniqueness of the geodesic with respect to the initial conditions. Besides, let $SM$ be the unit tangent bundle of $M$, that is, $$SM=\{(x,v): x \in M, v \in T_xM, \|v\|=1\}.$$ Since geodesics travel with constant speed, we have that $\varphi_t$ leaves $SM$ invariant. Given $(x,v) \in SM$, we obtain that $\varphi_t(x,v) \in SM$ for all $t\in \mathbb{R}$. It is well known that any closed, compact Riemannian manifold admits a complete geodesic flow. Finally, you may refer to the textbook \cite{A} for the details about the geodesic flow and the following results.

\begin{definition}
Let $(M^n,g)$ be a complete Riemannian manifold. We define the Liouville measure $L$ on $SM$. The measure $L$ is given locally by the product of the Riemannian volume on $M$ and the Lebesgue measure on the unit sphere. That is, for any subset $A=(U,A_x)\subset SM$, where $U\subset M$ is a subset of $M$, $A_x$ is a subset of the unit sphere of the tangent space at $x \in U$, $L$ is defined by
$$L(A)=\int_{x \in U} \int_{A_x} d\mathbb{S}^{n-1} dvol(x).$$
where $d \mathbb{S}^{n-1}$ is the usual Lebesgue measure on the unit sphere.
\end{definition}

A well known result related to the Liouville theorem is,
\begin{lemma}
Let $(M^n,g)$ be a complete manifold and $\varphi_t$ the geodesic flow. Then for any Borel set $B$ in $SM$ and $t\in \mathbb{R}$, we have,
$$L(\varphi_t(B)) = L(B).$$
That is, geodesic flow preserves Liouville measure.
\end{lemma}
Finally, we need the following basic integral form of the scalar curvature $Sc(g)$ in terms of the Ricci curvature $Ric(g)$.

\begin{lemma} Let $(M^n,g)$ be a complete Riemannian manifold. Then for any $p 
\in M$, we obtain,
\begin{equation}
    Sc_p=\frac{n}{\mathrm{vol}(\mathbb{S}^{n-1})}\int_{v\in \mathbb{S}^{n-1}} Ric_p(v) d v.
\end{equation}
\begin{proof}
Let  $\{e_i\}_{i=1}^{n} \in T_pM$ be an orthonormal coordinate such that $Ric_p(e_i)=\lambda_i e_i$. Then, 
for any $v\in \mathbb{S}^{n-1}$, $$v=x_ie_i.$$
Hence
$$\sum_{i=1}^n x_i^2=1, \ Sc_p= \sum_{i=1}^{n}\lambda_i,$$
and then
$$\int_{v \in \mathbb{S}^{n-1}} Ric_p(v)d v= \int_{x \in \mathbb{S}^{n-1}} \sum_{i=1}^nx_i^2\lambda^2_i dx = \sum_{i=1}^n\lambda^2_i \int_{x \in \mathbb{S}^{n-1}} x_i^2 dx = \frac{ \mathrm{vol}(\mathbb{S}^{n-1}) Sc_p}{n}.$$

\end{proof}
\end{lemma}

\bigskip

\subsection{Integral of Curvatures}

\begin{definition}
Let $(M^n,g)$ be a complete, non-compact manifold, $\gamma(t)=\exp(tv), t \geq 0, v \in T_p M$ is called a ray if it is minimal on every interval
$$d(\gamma(t), \gamma(s))=|s-t|,\  s,t >0,$$
and the unit vector $v$  is called a direction of $\gamma(t)$. Assume that $\gamma_i(t)=\exp(tv_i), v_i \in T_pM$ are rays, $\{\gamma_i(t)\}$ are independent and orthogonal if their directions $\{v_i\}$ are linearly independent and mutually orthogonal at $p$. Moreover, we define the Busemann function $B_{\gamma}(x)$ associated with any ray $\gamma(t)$
$$
    B_{\gamma}(x) = \lim_{t \rightarrow \infty}(t - d(x, \gamma(t))).
$$
\end{definition}
Here, $B_{\gamma}(x)$ is well-defined since $f(t)= t-d(x,\gamma(t))$ is increasing in terms of $t$ and uniformly bounded from above.
\vskip 2mm

\begin{lemma}\label{ricciestimate}
Let $(M^n,g)$ be a complete, non-compact manifold with a pole $p$ and $Ric(g) \geq 0$. Then, for any ray $\gamma(t)$ with $\gamma(0)=p$,
\begin{equation}
    \frac{1}{r}\int_0^r t^2 Ric(\gamma^\prime(t), \gamma^\prime(t))dt \leq n-1.
\end{equation}
In particular, as $n=3$, 
\begin{equation}
    \frac{1}{r}\int_{B(p,r)} Ric(\nu,\nu) (x)d\mathcal{H}^3(x)\leq 8\pi.
\end{equation}
where $\nu$ is the outer unit normal vector field of the geodesic sphere.
\begin{proof}
Assume that $\gamma(t)$ with $\gamma(0)=p$ is a ray and $V(t)$ is a smooth vector field along $\gamma(t)$, we consider the variation of $\gamma(t)$:
$$\gamma(t,s)= \exp_{\gamma(t)}(sV(t)), s \in \mathbb{R}.$$
Since $\gamma(t)$ is a minimizing geodesic, then, on any interval $[0,b]$, we obtain that the second variation of length functional is non-negative by \ref{stable}. i.e.,
$$\int_a^b |\nabla V(t)|^2 - K(\gamma^\prime(t), V(t))dt \geq 0.$$
Hence,
$$\int_a^b K(\gamma^\prime(t), V(t)) dt \leq \int_a^b |\nabla V|^2 dt.$$

Then, for any $t \in [0,b]$, we assume that $\{e_i(t)\}_{i=1}^{n-1}$ is the parallel vector field such that $\{e_i(t), \gamma(t)\}$ forms an orthonormal base in $T_{\gamma(t)}M$. Now we fix any $x \in (0,b)$ and then  take 
$$V_i(t) = \left\{\begin{array}{ll}
   \frac{t}{x}e_i(t),  & t \in [0,x]; \\
   
    \frac{b-t}{b-x}e_i(t), & t \in [x,b]. 
\end{array}
\right.$$
Then, we plug $V_i(t)$ into the inequality to obtain
\begin{equation*}
    \frac{1}{x^2} \int_0^x t^2 Ric (\gamma^\prime(t),\gamma^\prime(t) )dt + \frac{1}{(b-x)^2}\int_0^{b-x} t^2Ric(\gamma^{\prime}(b-t), \gamma^{\prime}(b-t))dt \leq (n-1)(\frac{1}{x}+ \frac{1}{b-x}).
\end{equation*}
By taking  $b \rightarrow \infty$ and the assumption that $Ric(g) \geq 0$, we obtain
$$\frac{1}{x^2} \int_0^x t^2 Ric (\gamma^\prime(t),\gamma^\prime(t) )dt \leq \frac{n-1}{x}.$$
Hence, for any $r >0$
\begin{equation} \label{riccionray}
    \frac{1}{r} \int_0^r t^2 Ric (\gamma^\prime(t),\gamma^\prime(t) )dt \leq n-1.
\end{equation}

Since $p \in M $ is a pole, we have for any $t >0$, the geodesic sphere $\partial B(p,t) \subset M$ is a sphere and on which the unit normal vector field is well-defined by $\nu$. As a result, we obtain $f(x)= Ric(\nu,\nu)(x)$ is a well-defined, induced function on each point $x \in M$. Finally, for any $v \in T_p M$, $\exp(tv)$ is a ray and hence we can obtain the estimate (\ref{riccionray}) on it. Then, by the volume comparison theorem $Area(\partial B(p,t)) \leq 4\pi t^2$ and coarea formula along the distance function, we obtain
$$\frac{1}{r}\int_{B(p,r)} Ric(\nu,\nu)d\mathcal{H}^3 \leq 4\pi (n-1).$$
Here, $\mathcal{H}^3$ is the 3 dimensional Hausdorff measure on $M^3$.
\end{proof}
\end{lemma}

\bigskip

Then, let's introduce the following type of geometrically relative Bochner formula.
Let $f$ be a smooth function defined on a complete Riemannian manifold $(M^n,g)$, we define the level set of $f$ as

$$L_t^f =\{x \in M: f(x)=t\}.$$

On each level set $L_a^f$, if  it is a smooth $n-1$ dimension embedded submanifold in $M$, we define the second fundamental form and mean curvature of $L_t^f$ by $A$ and $H$ respectively with respect to the outer normal vector field $u = \frac{\nabla f}{|\nabla f|}$. Hence,
$$A=\nabla_{L^f_t}(u),\ H=\text{Div}(u).$$
Here, $\nabla_{L^f_t}$ is the restriction of $\nabla $ on $L^f_t$. Then, we introduce $G= H^2 - |A|^2.$

Then, we have the following geometrically relative Bochner formula in \cite{pet}. Here, we give a different proof using Bochner formula of the vector field. 

\begin{lemma}\label{bf}
Suppose that $(M^n,g)$ is a complete Riemannian manifold and there exists no critical point in $[a,b]$ for $f$. Then,
\begin{equation}
    \int_{f^{-1}[a,b]} Ric(u, u) - G = \int_{L^f_a } H - \int_{L^f_b}H.
\end{equation}

\begin{proof}
By  Bochner formula of the vector field $X$ over $M$\cite{ma}, we have
\begin{equation}\label{X}
   \frac{1}{2}\Delta |X|^2 = |\nabla X|^2 + \text{Div}(L_Xg)(X)-Ric(X,X)-\nabla_X \text{Div}(X)
\end{equation}
Now, we take $X=u=\frac{\nabla f}{|\nabla f|}$ in (\ref{X}), then, $$Ric(u,u)= |\nabla u|^2 + \text{Div}(L_u g)(u) - \nabla_u \text{Div}(u).$$

For any $x \in L^f_t$, we obtain that $u(x) = \frac{\nabla f(x)}{|\nabla f|(x)}$ is the unit normal vector field of $L_t^f$ at the point of $x \in L^f_t$ and then we assume that $\{e_i,u\}_{i=1}^{n-1}$ forms an orthonormal base in $T_x M$, hence
$$|A|^2 =|\nabla u|^2,\ H=\text{Div}(u).$$

Moreover,  by using integration by parts over $f^{-1}([a,b])$ for $\nabla_u (\text{Div}(u))$, we obtain,
$$
\begin{array}{cl}
     & -\int_{f^{-1}([a,b])}\nabla_u (\text{Div}(u)) \\
     =& -\int_{f^{-1}([a,b])} \nabla_u H\\
     =& -\int_{f^{-1}([a,b])} \text{Div}(Hu)- H\text{Div}(u)\\
     =& -\int_{f^{-1}([a,b])} \text{Div}(Hu)-H^2\\
     =& \int_{f^{-1}([a,b])} H^2 +\int_{L^f_a} H- \int_{L^f_b} H.\\
\end{array}
$$

For the term $\text{Div}(L_u g)(u)$,
$$
\begin{array}{cl}
     &  \int_{f^{-1}([a,b])}\text{Div}(L_u g)(u)\\
     =& \int_{f^{-1}([a,b])}\nabla_{e_i} (L_u g)(e_i,u)\\
     =& \int_{f^{-1}([a,b])}\nabla_{e_i}(L_u g(u,e_i))- (L_ug)(e_i, \nabla_{e_i} u)\\
     =& \int_{f^{-1}([a,b])}\nabla_{e_i}((L_ug)u,e_i))- (L_ug)(e_i, \nabla_{e_i} u)\\
     =&-2\int_{f^{-1}([a,b])}|A|^2. \\
\end{array}
$$
Hence,
$$\int_{f^{-1}([a,b])} Ric(u,u) = \int_{f^{-1}([a,b])} G - \int_{L^f_b} H+  \int_{L^f_a} H$$

\end{proof}
\end{lemma}

\begin{remark}
From the perspective of function theory on a complete manifold: Sectional curvature would impose the condition on the Hessian of functions or the second fundamental form of level set. Ricci curvature would impose the condition on the Laplacian of functions or the mean curvature of the level set. Lemma \ref{bf} seems trivial, but it is deep for the author, since the integral of the second fundamental form can be expressed in terms of the mean curvature and the Ricci curvature on the ambient Riemannian manifold.
\end{remark}
\bigskip

\subsection{Width and Positive Scalar Curvature}

Let's introduce the result related to the positive scalar curvature to  our article. The study of non-negative and positive scalar curvature is a  very important topic in geometry analysis, which is pioneered by the works \cite{gromovlawson, SY} of Gromov-Lawson and Schoen-Yau. These works provide us two paths of the understandings of the geometry and topology of scalar curvature bounded below: spin techniques and minimal surface techniques. According to their works, it's known that $\mathbb{T}^n$ admits no complete Riemannian metric with non-negative scalar curvature unless it is flat.  In recent Gromov's work \cite{gromovfourlectures}, he introduces the $\mu$ bubble,  which is detailed by Zhu in his work \cite{zhu} where his result indicates that positive scalar curvature implies that $2$-systole is bounded above in terms of the lower bound of the scalar curvature. After that, many applications of $\mu$ bubble have been expanded to study the existence of Riemannian metric with positive scalar curvature \cite{cl,gromovaspherical}. Here, we will use the following result, which relates the size to the positive scalar curvature \cite{metricinequality}.

\bigskip
Suppose that $M^{n} = T^{n-1} \times I$ where $I$ is an interval $[a,b], a < b$. Here, $M$ is called a torical band. We define
$$\partial M= T^{n-1}\times \{b\} \bigcup T^{n-1} \times \{a\} = : \partial_+ M \bigcup \partial_-M.$$
then, 
\begin{equation}
    d(\partial_+ M, \partial_-M) =\inf_{x \in \partial_+ M, y \in \partial_- M } \{d(x,y)\}.
\end{equation}
Then, the following theorem holds

\begin{theorem}[G-WXY \cite{metricinequality, gromovfourlectures, wxy}]\label{toricalband}
Let $(M,g)$ be a $n$-dimensional torical band with $Sc(g) \geq n(n-1)$. Then,
\begin{equation}
    \ d(\partial_+ M, \partial_-M) <\frac{2\pi}{n}.
\end{equation}
\end{theorem}
Theorem \ref{toricalband} plays a vital role in the proof of Theorem \ref{mainvolumegrowth}. Combing it with the work \cite{MT} of  McLeod-Topping, we avoid analyzing the singular Ricci limiting space to achieve our goal.

\begin{theorem}[Spherical Lipschitz Bound Theorem \cite{metricinequality}] \label{slb}
Let $(M^n,g)$ be a Riemannian manifold(possibly incomplete) with $Sc(g) \geq n(n-1)$. Then, for all continuous maps $f$  from $M$ to 
  the unit sphere $\mathbb{S}^n$ (and also  to the hemisphere to $\mathbb{S}^n_+$(1) of 
  non-zero degrees, we have,
   $$Lip(f)> \frac  {c}{\pi\sqrt n}\mbox { for the above  }  c>\frac {1}{3}.$$
\end{theorem}
\vskip 2mm

Finally, the following lemma is also needed for the proof of Theorem \ref{mainvolumegrowth}.

\begin{lemma}\label{independentgeodesic}
Let $(M^n,g)$ be a complete, non-compact Riemannian manifold with $Ric(g) \geq 0$. If there exists a sequence of $p_i\rightarrow \infty$ and $R_i\rightarrow \infty$ such that $$\text{vol}(B(p_i,R_i))=c(R_i)R_i^{k}, n-1  \geq k \geq 1. $$
with $c(R_i) \rightarrow \infty$ as $i \rightarrow \infty$ and $vol(B(p,1)) \geq v >0$ for all $p \in M$,
then there exists a sequence $q_i \in M$ such that $(M,q_i)$ pointedly Gromov Hausdorff converges to a length space $(X\times \mathbb{R}^l, p_\infty)$ with $$l \geq k+1.$$
Hence, it implies that there exists at least $k+1$ rays $\{\gamma_l^{(i)}\}_{l=1}^{k+1}$ which are linearly independent and orthogonal at $q_i$ in $M$ such that for all $l=1,2,\cdots, k, k+1$, the length of $\gamma_{l}^{(i)}$, $L(\gamma_{l}^{(i)}) \rightarrow \infty$ as $i \rightarrow \infty$.

\begin{proof}
Since we assume that $Ric(g) \geq 0$, by the precompactness theorem and Cheeger-Colding theory \cite{cc}, we obtain that, up to subsequence,  $(M,p_i,vol_i)$ converges to a metric measured length space $(X,x_\infty, \mu_\infty)$ with a Borel measure $\mu_\infty $ on $X$. Moreover, since it is assumed that $vol(B(p,1)) \geq v >0$, for any geodesic ball $B(p_i,r) \subset M$ and $B(x_\infty, r) \subset X$, we have,
$$\lim_{i \rightarrow \infty } vol(B(p_i,r)) = \mu_{\infty}(B(x_\infty, r)).$$
Furthermore,  let $\gamma_i: [0,\infty) \rightarrow M$ be a ray with $\gamma_i(0)=p_i$, then we introduce that 
$$\sigma_i(t)=\gamma_i(t+R_i): [-R_i,\infty) \rightarrow M.$$
and we set $q_i = \sigma_i(0)$.
By the assumption of the volume growth, we have, 
$$vol(B(q_i,2R_i) \geq c(R_i)R_i^{k}.$$
Hence, we can replace $p_i$ by $q_i$ in  the precompactness theorem. So we have $\sigma_i$ will converge to a line $\sigma_\infty$ in $X$. By the splitting theorem in the Ricci limiting space \cite{cc},  we obtain
$$X=X_1\times \mathbb{R} \ \text{and} \ \mu_\infty = \mu_\infty^1\times \mathbb{R}.$$

If we let $q_i \rightarrow  q_\infty $ and  $R_i$ large, then
$$\mu_\infty (B(q_\infty, R_i) \geq c(R_i)R_i^k, c(R_i) \rightarrow \infty, \ \text{as}\  R_i \rightarrow \infty.$$
Here $c(r)$ may be different line by line.
Hence, for metric ball $B_1(q_\infty^1,R_i) \subset X_1$,
$$\mu_\infty^1 (B_1(q_\infty^1,R_i) \geq c(R_i)R_i^{k-1}, \text{and} \  c(R_i) \rightarrow \infty, \ \text{as}\  R_i \rightarrow \infty. $$
Here, $q_\infty^1 $ is from $q_\infty= (q_\infty^1,x^n), x^n \in \mathbb{R}$. Hence, we obtain that $(X_1, \mu_\infty^1)$ is a non-compact metric measured length space. Then, there exists a ray in $X_1$, we can retake our base point in $M$ to obtain a line associated with the ray as we did above by pulling back the ray to $M$. Hence, we have
$$(X_1=X_2\times \mathbb{R} , \mu_\infty^1=\mu_\infty^2 \times \mathbb{R}).$$
Finally, we continue this process $k-1$ times to obtain  the limiting space
$$(X_{k-1}=X_{k} \times \mathbb{R},\  \mu^k = \mu_\infty^{k-1} \times \mathbb{R})$$
and for any metric ball $B(q_\infty^k,R_i ) \subset X^k$, 
$$vol(B(q_\infty^k,R_i ))) \geq c(R_i) , , c(R_i) \rightarrow \infty, \ \text{as}\  R_i \rightarrow \infty.$$
Here $q_\infty^{k-1}=(q_\infty^k, x^{n-k+1}), x^{n-k+1} \in \mathbb{R}$. Hence, $X_k$ is still non-compact, otherwise, its volume should be finite. Hence, by the same argument above, we obtain that $X_k$ splits as $X_{k+1}\times \mathbb{R}$. Hence, we finally find a  sequence  $q_i \in M$ such that $(M,q_i)$ pointedly Gromov Hausdorff converges to $X\times \mathbb{R}^{k+1}$.  We complete the proof of the lemma. 
\end{proof}
\end{lemma}

\subsection{Proof of Proposition \ref{completegreen}}

Before we are going to the  proof, let's first see the compact case: Setting $l=\mathrm{conj}(M)$, we consider
any geodesic ball $B(m,l), m \in M$ and any $q \in \partial B(m,l)$, there exists arc length parameter $\gamma: [0,l] \rightarrow M$ which is the shortest geodesic connecting $m,q$ . Then, we consider the index form for any variational vector $X$ of $\gamma(t)$ with $X(0)=X(l)=0$ 
$
0 \leq I(X,X)= \int_0^l \|\nabla X\|^2 - \langle R(X, \gamma^\prime(t))\gamma^\prime(t), X\rangle ds.
$ Hence,
\begin{equation}
     \int_0^l \langle R(X, \gamma^\prime(t))\gamma^\prime(t), X\rangle ds \leq \int_0^l \|\nabla X\|^2ds.
\end{equation}
If we pick $X=sin(\frac{\pi}{l}t)\nu$ with $\nu = \nu(t)$ a parallel unit vector field along $\gamma(t)$, then
\begin{equation}
    \int_0^l sin^2(\frac{\pi}{l}t)\langle R(\nu, \gamma^\prime(t))\gamma^\prime(t), \nu \rangle dt \leq  (\frac{\pi}{l})^2\int_0^l sin^2(\frac{\pi}{l})dt.
\end{equation}
By a direct calculation, we obtain
$ (\frac{\pi}{l})^2 \int_0^l sin^2(\frac{\pi }{l} t) dt= \frac{\pi^2}{2l}$ and then
$$ \int_0^l sin^2(\frac{\pi}{l}t)\langle R(\nu, \gamma^\prime(t))\gamma^\prime(t), \nu \rangle dt \leq \frac{\pi^2}{2l}.$$

$$\int_0^l sin^2(\frac{\pi}{l}t) Ric(\gamma^\prime(t), \gamma^\prime(t))dt \leq \frac{(n-1)\pi^2}{2l} .$$
Integrating over the unit tangent bundle $SM$, we have,
\begin{equation} \label{stableinequality}
   \int_{SM} \int_0^l sin^2(\frac{\pi}{l}t) Ric(\gamma^\prime(t), \gamma^\prime(t))dt dL \leq  \int_{SM} \frac{(n-1)\pi^2}{2l} d L .
\end{equation}
For the integral on the left, we have,

\begin{align*}
    & \ \ \ \int_{SM} \int_0^l sin^2(\frac{\pi}{l}t) Ric(\gamma^\prime(t), \gamma^\prime(t))dt dL
\\
&= \ \int_0^l \int_{SM} sin^2(\frac{\pi}{l}t) Ric(\gamma^\prime(t), \gamma^\prime(t)) dLdt\\
&= \ \int_0^l \int_{m\in M} \int_{v \in S_mM} sin^2(\frac{\pi}{l}t) Ric((\varphi_{t})_\star v, (\varphi_{t})_\star v)dv d\mathrm{vol}(m) dt \\
&= \ \int_0^l sin^2({\frac{\pi}{l}}{t}) dt\int_{m \in M}\int_{S_m M}Ric(v)dvd\mathrm{vol}(m)\\
& = \ \frac{\mathrm{vol}(\mathbb{S}^{n-1})}{n} \int_0^l sin^2({\frac{\pi}{l}}{t})dt \int_{m \in M} Sc(m)  d \mathrm{vol}(m) \\
&\geq \ \frac{(n-1)l }{2} \mathrm{vol}(\mathbb{S}^{n-1}) \mathrm{vol}(M) \ \ \text{since $Sc \geq n(n-1)$}. \\ 
\end{align*}
By a direct calculation, we obtain that
$$\int_{SM} \frac{(n-1)\pi^2}{2l} d L = \frac{(n-1)\pi^2}{2l}\mathrm{vol}(\mathbb{S}^{n-1})\mathrm{vol}(M).$$
Hence, $$l \leq \pi.$$
Finally, if $l=\pi$, then all above inequalities are equalities. Hence, $M$ has constant sectional curvature $K=1$ with $\mathrm{Diam(M)}=\mathrm{Inj}(M)= \pi$. Hence, by Theorem \ref{mdt}, we obtain that $M$ is isometric to the round sphere $\mathbb{S}^{n}$.

Now, let's come back to the proof of Proposition \ref{completegreen}:  Rather than integrating over the unit vector bundle $SM$, we consider the geodesic ball $B=B(p,r) \subset M$ and $B_{-l}= B(p, r-l)$. Then we start from inequality (\ref{stableinequality}),
$$   \int_{SB} \int_0^l sin^2(\frac{\pi}{l}t) Ric(\gamma^\prime(t), \gamma^\prime(t))dt dL \leq  \int_{SB} \frac{(n-1)\pi^2}{2l} d L .$$

\begin{align*}
    & \ \ \ \int_{SB} \int_0^l sin^2(\frac{\pi}{l}t) Ric(\gamma^\prime(t), \gamma^\prime(t))dt dL
\\
&= \ \int_0^l \int_{SB} sin^2(\frac{\pi}{l}t) Ric(\gamma^\prime(t), \gamma^\prime(t)) dLdt\\
&\geq \ \int_0^l sin^2({\frac{\pi}{l}}{t}) dt\int_{m \in M}\int_{S_m B_{-l}}Ric(v)dvd\mathrm{vol}(m), Ric(g) \geq 0 \ \text{on}\ B(p,R)\\
& = \ \frac{\mathrm{vol}(\mathbb{S}^{n-1})}{n} \int_0^l sin^2({\frac{\pi}{l}}{t})dt \int_{m \in B_{-l}} Sc(m)  d \mathrm{vol}(m). \\
\end{align*}
Hence, $$\int_{B(p, r-l)} Sc \leq n(n-1) \frac{\pi^2}{l^2}vol(B(p,r)).$$

Moreover, if we assume that $Ric(g) \geq 0$ and $Sc(g) \geq n(n-1)$ on $M$, then for any $r > l$, 
$$\frac{\mathrm{vol}{B(p,r-l)}}{\mathrm{vol}(B(p,r))} \leq \frac{\pi^2}{l^2}.$$
By volume comparison theorem,  we obtain,
$$\lim_{r \rightarrow \infty}\frac{\mathrm{vol}{B(p,r-l)}}{\mathrm{vol}(B(p,r))}=1.$$
Therefore $l \leq \pi$. We complete the proof of Proposition \ref{completegreen}.

As a corollary, we have, 
\begin{corollary}
Let $(M^n, g)$ be a complete, non-compact manifold with $Ric(g) \geq 0$ with $inj(M)=c>0$. Then, for any $p \in M$,
$$\frac{1}{\mathrm{vol}(B(p,c))}\int_{B(p,c)}Sc \leq \frac{2^nn(n-1)}{c^2}$$
\end{corollary}
From the perspective of Cheeger-Colding theory and Anderson's $C^\alpha$ convergence, it is too strong that we  assume that the injectivity radius has a uniformly lower bound. But, if you pay more attention to the generalized scalar curvature on Ricci limiting space, we still do not have a systematic way to introduce a useful scalar curvature on this singular space. Probably, this inequality may help study the Ricci limiting space for non-collapsing case in the future.

\section{Proof of Theorems}
In this section, we will prove Theorem \ref{polecase}, \ref{mainvolumegrowth} and \ref{width}. For the proof of Theorem \ref{polecase}, we will use the geometrically relative Bochner formula along the distance function and the stability of ray. For the proof of Theorem \ref{mainvolumegrowth}, we will combine the Gromov-Hausdorff convergence with the estimate of torical band to obtain the volume estimate. For the proof of Theorem \ref{width}, we analyze the level set of Busemann function to obtain the existence of function required.






\bigskip

\subsection{Proof of Theorem \ref{polecase}}

\begin{theorem}
Let $(M^3,g)$ be a complete, non-compact three-dimensional Riemannian manifold with a pole $p$ and $Ric(g) \geq 0$. Then
\begin{equation}
    \limsup_{r \rightarrow \infty } \frac{1}{r} \int_{B(p,r)}Sc \leq  20\pi. 
\end{equation}

\begin{proof}
Let's first define $f(x):=d(p,x)$, on each level set $L_t^f$, we have the following type of Gauss equation called S-Y trick on minimal surface \cite{SY}
$$2\Bar{K}= Sc-2Ric(\nu,\nu) + G.$$
Here, $\bar{K}$ is the Gauss curvature of the level set  $L_t^f$.
Hence,
$$Sc= 2\Bar{K}+2Ric(\nu,\nu) -G.$$
Then integrating it over $B(a,b)=B(p,a,b)$, we obtain,
$$\int_{B(a,b)} Scd\mathcal{H}^3= \int_{B(a,b)} (2\Bar{K} + 2Ric(\nu,\nu) -G)d\mathcal{H}^3.$$
By Lemma \ref{bf}, we obtain,
$$ \int_{B(a,b)} G d\mathcal{H}^3 = \int_{B(a,b)} Ric(\nu,\nu) d\mathcal{H}^3 - \int_{L^f_b} H +\int_{L^f_a}H .$$
Hence,
$$\int_{B(a,b)} Sc \ d\mathcal{H}^3=2 \int_{B(a,b)} \bar{K} d\mathcal{H}^3+ \int_{B(a,b)} Ric(\nu,\nu) d\mathcal{H}^3 +\int_{L^f_b} H -\int_{L^f_a}H.$$

By the coarea formula and Gauss Bonnet theorem on each level set surface, we have,
$$ \int_{B(a,b)} \bar{K} d\mathcal{H}^3 = \int_a^b \int_{\partial B(r)} \bar{K} d\mathcal{H}^2 d r = \int_a^b 2\pi \chi(\partial B(r))dr= 4\pi(b-a).$$
Here, we used  $\partial B(r)$ is a topological sphere for any $r \in [a,b]$. Furthermore, by the volume comparison theorem and its proof, we have,
$$\int_{L^f_b} H  \leq 4\pi b. $$
Combining above all estimate and Lemma \ref{ricciestimate}, we get
$$\int_{B(a,b)} Sc d\mathcal{H}^3 \leq 8\pi(b-a) + 8\pi(b-a) + 4\pi b- \int_{L^f_a}H. $$
By taking $b \rightarrow \infty$ and then $a \rightarrow 0$, we have,
$$\limsup_{r \rightarrow \infty} \frac{1}{r}\int_{B(p,r)} Sc \ d\mathcal{H}^3 \leq  20\pi. $$
\end{proof}
\end{theorem}

\subsection{Proof of Theorem \ref{mainvolumegrowth}}

\begin{theorem}
Let $(M^3,g)$ be a complete, non-compact three-dimensional Riemannian manifold with $Ric(g) \geq 0$ and $Sc(g) \geq 6$. Then, for any $p \in M$, we obtain,
\begin{equation}
\limsup_{R \rightarrow \infty} \frac{\text{vol}(B(p,R))}{R} < \infty,
\end{equation}
provided that $\text{vol}(B(q,1)) \geq \epsilon >0$ for all $q$.
\end{theorem}

\begin{proof}
Let's show that it is sufficient to prove that there exists one $p \in M$ such that
\begin{equation}\label{onepoint}
\limsup_{R \rightarrow \infty} \frac{\text{vol}(B(p,R))}{R} < \infty.
\end{equation}
We assume that the estimate (\ref{onepoint}) holds for $p$, then for any $q \in M$, we have
$$\frac{\text{vol}(B(q, r))}{r} \leq  \frac{\text{vol}(B(p, r+2d(p,q)))}{r} = \frac{\text{vol}(B(p, r+2d(p,q)))}{r+2d(p,q)} \frac{r+2d(p,q)}{r}.$$
Then, by taking the $\limsup$ on both sides, we have,
$$\limsup_{R \rightarrow \infty} \frac{\text{vol}(B(q, R))}{R} < \infty.$$

Now, let's prove inequality (\ref{onepoint}) by contradiction argument: Suppose that there exists a sequence of $p_i\in M,R_i \in \mathbb{R}$ such that $p_i \rightarrow \infty$ and $R_i\rightarrow \infty$ and
$$\frac{\text{vol}(B(p_i,R_i))}{R_i} \rightarrow \infty, i \rightarrow \infty.$$
By the Lemma \ref{independentgeodesic}, there exists at least $2$ rays that are independent and orthogonal at each $p_i$. By the main theorem in \cite{MT} and Cheeger-Colding theory \cite{cc}, there exists a subsequence $\{q_i\}$ of $\{p_i\}$ such that
\begin{equation}\label{gseq}
    (M^3,p_i,g_i) \rightarrow (M_\infty, p_\infty, d),
\end{equation}
in the sense of Gromov-Hausdorff convergence with the following properties ($**$):
\begin{itemize}
    \item $M_\infty$ is a smooth manifold. Notice that the topological regularity is only known for $n=3$. Here, smooth manifold means that $M_\infty$ is a smooth differential manifold topologically, and we do not know anything about the deep metric structure of the Ricci limiting space. In fact, we mainly use the topological structure in our paper;
    \item For any $p_\infty$, there exists a sequence of smooth maps $$\varphi_i : B_d(p_\infty, i) \rightarrow M_i,$$
    such that $B_d(p_\infty, i)$ is diffeomorphic onto $\varphi_i(B_d(p_\infty, p))$ and $\varphi_i(p_\infty) =q_i$. Here $B_d(p_\infty,i)$ is the metric ball in $(M_\infty, d)$ with radius $i$ and center $p_\infty$;
    \item Under the above item, for any $R>0$, $$d_{g_i}(\varphi_i(x),\varphi_i(y)) \rightarrow d(x,y),$$
    uniformly on $B_d(p_\infty, R)$ as $i \rightarrow \infty$. Hence, the convergence is at least $C^0$ convergence only in the sense of metric space. 
\end{itemize}

\begin{remark}

For the notation used in \ref{gseq}: actually, $g_i =g$,  we write it as $g_i$  since we want to match $g_i$ with the base point $p_i$. Moreover, we do not know if the Riemannian metric $g_i$ will $C^0$-converge to a smooth Riemannian metric.
\end{remark}

\vskip 2mm

Moreover, for the limiting space $M_\infty$, we have the following two cases by Lemma \ref{independentgeodesic}: $M^1 \times \mathbb{R}^2$, $\mathbb{R}^3$.
\begin{itemize}

   \item \textbf{Case 1:} If 
   $$M_\infty \simeq M^1\times \mathbb{R}^2.$$
   isometrically, we will have $M^1$ is one dimensional, topologically, smooth manifold. This implies that $M^1=\mathbb{S}^1$. Hence, $M_\infty \simeq  \mathbb{S}^1 \times R^2$. Here, we merely obtain that the manifold is smooth in the sense of topology. However, we did not know if the metric $d_\infty $ on $M^1\times \mathbb{R}^2$ is induced by a smooth Riemannian metric. Now we can overcome these difficulties as follows.
   
   On the limiting space $\mathbb{S}^1\times \mathbb{R}^2$, we consider the set
   $$T=B_d(p_\infty, 2R)-B_d(p_\infty. R).$$
   As $R$ is a large, fixed number, i.e., $R \geq 100$,   $T$ is $T^2 \times [R,2R]$ topologically and
   $d(\partial_+ T, \partial_-T)=R$ under the metric $d$. Since we know that, as  $i \geq 10R$, $T \subset B_d(p_\infty, i)$ and $B_d(p_\infty, i)$ is diffeomorphic onto $\varphi_i(B_d(p_\infty, i))$, we reach that $\varphi_i(T)$ is a torical band in $M_i$ with a minor damage on the band distance. Moreover, we can always perturb $\varphi_i(T)$ to $K$ such that $K$ becomes a smooth manifold and its topology is kept fixed and 
   $$d_{g_i}(\partial_+K, \partial_- K) \geq \frac{1}{2}R.$$ 
   Hence, we obtain a compact Riemannian manifold with boundary
   $$(K=\mathbb{T}^2 \times I, g_K =g|_{K}, Sc(g_K) \geq 6, d(\partial_+K , \partial_-K) \geq 25).$$
   This contradicts with the Gromov's torical band estimate Theorem \ref{toricalband}.
   
   \bigskip
    \item \textbf{Case 2:} Otherwise,  by Cheeger-Colding theory in \cite{cc},
    $$M_\infty \simeq \mathbb{R}^3,$$ 
    isometrically. Since the limiting space is $\mathbb{R}^3$, we can always pick a big torical band in $\mathbb{R}^n$ and then proceed the same argument in case $1$ to reach a contradiction. Here, we will not repeat the argument again since it is totally the same as case $1$.  In fact, we may also use the argument in the proof of volume non-collapse in the Corollary \ref{bnvolumegrowth} below.

\end{itemize}
Together with all arguments above, we proved that for any $p \in M^3$, 
$$\limsup_{R \rightarrow \infty} \frac{\text{vol}(B(p,R))}{R} < \infty.$$
This completes our proof.
\end{proof}

\begin{remark}
In fact, by the proof, we see that the manifold with non-negative Ricci curvature, positive scalar curvature and volume non-collapse is asymptotic to $\mathbb{S}^2 \times [R, \infty)$ at infinity or splits globally. In fact, this can be also seen from the perspective of the minimal surface argument.
\end{remark}
By the proof of Theorem \ref{mainvolumegrowth}, we have the following weaker version of volume growth in higher dimension.

\subsection{The Proof of Corollary \ref{bnvolumegrowth}}

\begin{proof}

\begin{itemize}
\item \textbf{Proof of volume non-collapsed case}

As the proof of Theorem \ref{mainvolumegrowth}, we assume that the result does not hold. Hence, there exists a sequence $p_i \in M$ such that $(M, p_i)$ pointedly Gromov Hausdorff converges to $\mathbb{R}^n$. However, we do not know if the convergence is $C^0$ convergence or not in the sense of Riemannian metric convergence. Hence, we can not directly use Gromov's upper semi-continuity of scalar curvature under Riemannian metric $C^0$ convergence. Instead, we make use of an estimate in \cite{metricinequality}

First,  for all $\epsilon > 0$, there exists a Lipschitz map $f$ from $\mathbb{R}^n$ to the standard unit sphere $\mathbb{S}^n$ with $deg(f)\geq 1$ , $Lip(f) \leq \epsilon$ and $f$ is constant at infinity. Namely, $f$ is constant on $B^c(R) \subset \mathbb{R}^n$. Then, since we have, $(M_i, r_i)$ converges to $\mathbb{R}^n$ with respect to the Gromov Hausdorff convergence. Hence, there exists a map $\varphi_i : (M_i, r_i) \rightarrow (\mathbb{R}^n,0) $ with Lipschitz constant $Lip(\varphi_i) \leq 2$ and $B(2R) \subset Im(\varphi_i({B(r_i, 3R)}) $ for large $i$. Here $B(r_i, 2R)$ is the geodesic ball in $M_i$ centered at $r_i$. Finally, we construct a map $F_i= f\circ \varphi_i : M_i \rightarrow \mathbb{S}^n$ with $Lip(F) \leq 2\epsilon$ with $Sc(M_i) \geq n(n-1)$ and $deg(f_i) \geq 1$. If we pick  $\epsilon$ small enough, then what we obtain  contradicts with the Spherical Lipschitz Bound Theorem \ref{slb} (cited from \cite{metricinequality}). Hence,
$$\limsup_{R \rightarrow \infty} \frac{vol(B(p,R))}{R^{n-1}} < \infty.$$

\item \textbf{Proof of injectivity radius non-collapsed case}

As the proof of Theorem \ref{mainvolumegrowth}, we assume that the result does not hold. Hence, there exists a sequence $p_i \in M$ such that $(M, p_i)$ $C^\alpha, \alpha \in (0,1)$ converges to a smooth manifold $\mathbb{S}^1 \times \mathbb{R}^{n-1}$ since we assume that the injectivity radius has a uniformly positive lower bound \cite{ac}. Then, we take a torical band $T^{n-1} \times [R,2R]$ in $\mathbb{S}^1 \times \mathbb{R}^{n-1}$ for large $R$. Since the convergence is $C^\alpha$, we have that the properties $(**)$ are automatically satisfied by Anderson's result in \cite{ac}. Hence, we will reach a contradiction, since the following steps follow the same argument in the proof of Theorem \ref{mainvolumegrowth}, we have
$$\limsup_{R \rightarrow \infty} \frac{vol(B(p,R))}{R^{n-2}} < \infty.$$

\end{itemize}

\end{proof}

\subsection{Proof of Theorem \ref{width}}
\begin{proof} Since we assume that $(M^3,g)$ has nonnegative Ricci curvature, we have $(M,g)$ has at most $2$ ends.
\begin{itemize}
    \item If $(M^3,g)$ has $2$ ends, we have $(M,g)$ is split. i.e. $$M^3=\mathbb{S}^2 \times \mathbb{R}.$$ 
    In this case, we take the function $f$ as the projection $\mathbb{S}^2 \times \mathbb{R}$ to $\mathbb{R}$. Since $Sc(g) \geq 2$, it is trivial that for any $r \in \mathbb{R}$. $diam(f^{-1}(r)) \leq 4\pi$ and $f^{-1}(r)$ is a $2$ sphere and hence connected;
    
    \item If $(M^3,g)$ has $1$ end,  we take any ray $\gamma(t) \in M$ and obtain the associated Busemann function
    $$B_{\gamma}(x): M \rightarrow \mathbb{R}.$$
     \textbf{Claim:} $f(x) = B_{\gamma}(x)$ is a continuous function as required.

    
    Assume that there exists a sequence $r_i \rightarrow \infty$ such that, $diam(f^{-1}(r_i)) \rightarrow \infty.$
    By the proof of Theorem \ref{mainvolumegrowth}, there exists a subsequence $p_i \in f^{-1}(r_i) $ such that $(M, p_i)$ pointedly Gromov Hausdorff converges to a length space(smooth manifold) $(M_\infty=X^2_\infty \times \mathbb{R}, p_\infty, d_\infty)$ and $X_\infty$ is a compact manifold. Then, we take a large metric ball $B(p_\infty, 10R) \subset M_\infty$ such that for large $i$, $5R \geq diam(f^{-1}(r_i)) \geq R$ but the level set $f^{-1}(r_i)$ is contained into some neighborhood of $\gamma$: $N_s(\gamma)= \{x \in M, d_g(x, \gamma) \leq s\}$. Here, $s$ only depends on $diam(X_\infty^2)$. 
    
    Since we assume that $\text{diam}f^{-1}(r_i)$ diverges to $\infty$ and we know $\gamma(r_i) \in f^{-1}(r_i)$, then, for the large $i$ picked above, we take a point $y_i \in f^{-1}(r_i)$ such that $d_g(y_i, \gamma(r_i))$ is large. Moreover, we pick $q_i \in \gamma$ such that $d(y_i, \gamma)= d(q_i, y_i) $ that is small relative to the $r_i$.  Hence, for any large $t \geq r_i$, by the definition of Busemann function, we have 
    $$r_i= f(y_i)= B_\gamma(y_i) > t-d(y_i,\gamma(t).$$
    If we initially pick $i$ large enough such that $d_{GH}(B(p_i,10R), B(p_\infty, 10R)$ is small, we obtain that for large $t$
    $$r_i\geq  t-d(y_i,\gamma(t)) > r_i+1.$$
    Hence, we reach a contradiction. We conclude that there exists a constant $c$ such that $diam(f^{-1}(r)) \leq c$.
\end{itemize}
\end{proof}

\begin{remark}
Geometrically, the proof is very clear. If we keep  $X_\infty \times \mathbb{R}$ in mind, it would be natural to argue the level set is uniformly bounded even the proof seems indirect.
\end{remark}

\bigskip

\begin{remark}
This article is a revised version of my original article, which was finalized in May 2021 and submitted to a Journal in September 2021.
\end{remark}

\end{document}